\documentclass[reqno,11pt]{amsart}
\usepackage{amscd,amssymb, amsmath}
\usepackage{graphicx}
\usepackage[frenchb, russian, english]{babel}
\usepackage[utf8]{inputenc}

\newtheorem{theorem}{Theorem}[section]
\newtheorem{lemma}[theorem]{Lemma}
\newtheorem{corollary}[theorem]{Corrolary}
\theoremstyle{definition}
\newtheorem{definition}[theorem]{Definition}

\newtheorem{proposition}[theorem]{Claim}

\newtheorem{remark}[theorem]{Note}
\newtheorem{note}[theorem]{Note}

\DeclareMathOperator{\Hom}{Hom}
\DeclareMathOperator{\Sing}{Sing}

\DeclareMathOperator{\reg}{reg}

\DeclareMathOperator{\dist}{dist}

\begin{document}
\author{Lev Soukhanov \\ \tiny{National Research University Higher School of Economics}}
\title{Non-Collapsible Dual Complexes and Fake del Pezzo Surfaces}
\pagestyle{plain}
\maketitle
\begin{abstract}
We propose the new construction of complex surfaces with $h^{1, 0} = h^{2, 0} = 0$ from smoothings of normal crossing surfaces with non-collapsible dual complexes and carry it out for the simplest case of the duncehat complex, obtaining the surface with $h^{1,1} = 9$ (presumably Barlow surface).
\end{abstract}

\section{Introduction}
Let $X^{t}$ be a family of complex surfaces with smooth total space and smooth general fibers, and $X^{0} = X = \coprod_{i=1}^n X_i$ be a normal crossing surface. Suppose also that normalization of components $X_i$ are rational surfaces, and their intersection curves are also rational.

Let us denote by $V$ the set of $X_i$'s, by $E$ the set of (normalized) intersection curves (irreducible components of either $X_i \cap X_j$ or $Sing(X_i)$), by $T$ the set of triple points.

\begin{definition} \textbf{Dual complex} $\Delta_X$ of X is a $2$-dimensional triangulated complex $(V, E, T)$.
\end{definition}

\begin{remark} Dual complex is simplicial (has no multiple edges / triangles) if and only if $X$ is simple normal crossing.
\end{remark}

Dual complexes are considered in the various degeneration problems, starting from the famous Kulikov's degenerations of K3 surfaces (\cite{Ku}, \cite{PP}, \cite{Fr}). Lot of work was also put on the dual complexes of degenerations of del Pezzo surfaces and Fano varieties in general (\cite{Ka}, \cite{Fu}, \cite{Tz}). The fundamental work of Kollar, de Fernex and Xu \cite{dFKX} considers the topology of dual complexes of resolutions of singularities. The (higher dimensional) case of Calabi-Yau and hyperkahler varieties is the subject of the vast Gross-Siebert program (\cite{GS}).

\begin{proposition} $H^{*}(\Delta_X, \mathbb{C}) \cong H^{*}(X, \mathcal{O})$
\end{proposition}

\begin{proof}
The sheaf $\mathcal{O}$ admits the following resolution:
\[0  \rightarrow \mathcal{O}_X \rightarrow \bigoplus_{X_i \in V} \mathcal{O}_{X_i} \rightarrow \bigoplus_{C \in E} \mathcal{O}_{C} \rightarrow \bigoplus_{p \in T} \mathcal{O}_{p}  \rightarrow 0 \]

Cohomologies of the sheaves of this resolution are zero except of $h^0 = 1$ due to rationality of the components, which proves the claim.
\end{proof}

\begin{proposition}
$H^{*}(\Delta_X, \mathbb{C}) \cong H^{*}(X, \mathcal{O}) \simeq H^{*, 0} (X^{t})$
\end{proposition}

\begin{proof}
It follows from the results of Schmid and Steenbrink, see proposition 2.16 from \cite{St}.
\end{proof}

The dual complex captures also some important algebro-geometric information about both $X$ and general fiber $X_{t}$. One of the important things is that while running relative MMP on the family $X_t$, the contractions change the structure of $\Delta_X$ in a controllable way - they correspond to \textbf{elementary collapses}.

\begin{definition} \textbf{Elementary collapse} is an operation on a triangulated complex which removes a simplex and one of its faces. It is required that the simplex being removed is not a facet to anything, and its facet is free (not a facet of any other simplex).
\end{definition}

The complex being collapsible is much stronger than being contractible. There are many examples known in geometric topology, simplest ones being dunce hat and Bing's house with two rooms.

\begin{theorem}(Theorem $4$ of \cite{dFKX})

Suppose the general fiber $X_t$ is rational. Then, $\Delta_X$ is collapsible.
\end{theorem}

In \cite{dFKX} it is stated that the authors do not know any examples of non-collapsible yet contractible dual complexes. The main point of this paper is constructing such an example from a dunce hat complex. We use, essentially, smoothing techniques from \cite{Fr} and \cite{KN}.

We also believe it is very natural to consider such an examples: they are automatically surfaces with $h^{1, 0} = h^{2, 0} = 0$, and while author didn't manage to prove that they are simply connected, the fundamental group probably can be analyzed too, and search of such surfaces is a longstanding problem.

We consider this text as more of a proof of concept - while the construction is quite nice, the proofs in this text use some arguments about the structure of the moduli spaces which are not clean. We hope that most of the construction is possible to generalize, but did not manage to do it yet.

I would like to thank quite a lot of people. This work would not be possible at all without the contribution of Konstantin Loginov, who familiarized me with the works of Y. Kachi and N. Tziolas and guided through the Mori theory (which I'm quite shaky at). Discussions with Denis Teryoshkin, Dmitry Sustretov, Andrei Losev, Vadim Vologodsky and Dmitry Korb also helped me a lot. I would also like to thank Mikhail Kapranov for his suggestion on the connection of this work with possibly constructing phantom categories combinatorially (connection which is not yet clear but clearly important).

I also would like to specially thank participants of Lutsinofest for the discussions which helped me to figure out the gap in the first version of the proof, and Kostya Ivanov for the pictures of cubics.

The study has been funded by the Russian Academic Excellence Project '5-100'.

\section{Recollection on smoothing}
The question on the possibility of smoothing of the normal crossing variety, we believe, starts with the work of Friedman \cite{Fr}. He introduced the following notion:

\begin{definition}
The normal crossing variety $X$ is called $d$-semistable if its first tangent cohomology sheaf $T_X^1$ is isomorphic to $\mathcal{O}_{Sing(X)}$.
\end{definition}

\begin{remark}
For simple normal crossing variety \[T_X^1 \simeq \Hom(\bigotimes_i I_{X_i} / I_{X_i}I_{\Sing(X)}, \mathcal{O}_{\Sing(X)})\] For normal crossing varieties this identification works locally. Hence, $T_X^{1}$ is always a line bundle over the set $\Sing(X)$. For a surface, the restriction of $T_X^{1}$ on the component $C = X_i \cap X_j$ can be identified with $N_{C, X_i} \otimes N_{C, X_j} (\sum_k p_{ijk})$, where $p_{ijk} \in T$, vanishing of the $c_1$ of this line bundle is known as ''triple points formula''.
\end{remark}

The $d$-semistability is not sufficient for guaranteeing the existence of the smoothing. The corresponding obstruction is governed by log deformations theory. The source we are using is the paper of Kawamata and Namikawa \cite{KN} (it is not written in the modern language of log schemes, but has an advantage of being in the analytic category and is fine for our purposes).

\begin{definition}[Kawamata-Namikawa]
The \textbf{log atlas} on a normal crossing variety $X$ of dimension $n-1$ is a collection of charts $\{U_i\}$ on the singular locus $\Sing(X)$ together with holomorphic coordinate functions $u_1^i, ... u_n^i$, such that $U_i$ is identified to the open subset of the variety given in $\mathbb{C}^n$ by equations $u_{s_1}^i \cdot ... \cdot u_{s_k}^i = 0$ for some subset $\{s_1 ... s_k\} \in \{1...n\}$, and choice of permutations $\sigma^{ij}$ and holomorphic functions $z_k^{ij}$ on $U_i \cap U_j$, such that $z_k^{ij} u_k^j = u_{\sigma^{ij}(k)}^i$ (the choice of $z$'s is non-unique due to both functions sometimes vanishing on some irreducible components, yet unique on a singular locus). It is required that $\prod_{k} z_k^{ij} = 1$.

Two log atlases are called equivalent if their union is a log atlas. The classes of equivalence of log atlases are called  semistable log structures.
\end{definition}

\textit{Terminological note - in the modern language, these objects should probably be called semistable morphisms but we are going to use terminology from \cite{KN}}.

\begin{proposition}[Kawamata-Namikawa]
Variety is $d$-semistable if and only if it admits a semistable log structure.
\end{proposition}

\begin{proof}
This proposition is implicitly contained already in Friedman's work. While we won't reproduce the full proof here, we sketch one of the directions, needed further:

being $d$-semistable means the triviality of the line bundle $\bigotimes_i I_{X_i} / I_{X_i} I_{\Sing(X)}$ on a singular locus. The collections of $u^k_i$ define such trivializations on $U_k$.
\end{proof}

Semistable log structure could be thought about as ''zero-order'' smoothing deformation of $X$. In particular, it allows one to define logarithmic differential forms and logarithmic tangent bundle.

\begin{definition} The bundle of \textbf{logarithmic vector fields} $T_X(log)$ is defined as follows. The logarithmic vector field $v$ is a vector field on $X$ which preserves the semistable log structure (i.e. its flow moves atlas to an equivalent one). In the concrete terms it means that $\frac{1}{u^k_1}\frac{\partial u^k_1}{\partial v} + ... + \frac{1}{u^k_n}\frac{\partial u^k_n}{\partial v} = 0$.
\end{definition}

\begin{definition} Logarithmic deformation of $X$ is a deformation family $X^{t}$ endowed with the atlas over the $Sing(X^{0})$ such that $u^k_1 \cdot ... \cdot u^k_n = t$ and its restriction on $X^{0}$ defines the semistable log-structure of $X$.
\end{definition}

\begin{proposition}[logarithmic deformation theory]
The obstructions to the log deformation of $X$ lie in $H^2(X, T_X(log))$. The thickenings from order $k$ deformations to order $k+1$ deformations form a torsor overs $H^1(X, T_X(log))$.
\end{proposition}

We are now ready to present the main example of our consideration.

\section{Example}
\begin{definition} We will call the collection of $X_i$'s with the preimages of intersection curves $C_{ij}$ and gluing maps $\phi_{ij}: C_{ij} \rightarrow C_{ji}$ the \textbf{construct}. The dual complex should be thought as the assembly instructions for this construct. We also remark that the curves $C_{ij}$ correspond to the half-edges of the dual complex, and dual graph of the collection of curves $C_{ij}$ on $X_i$ corresponds to the link of the vertex $X_i \in V$. The construct can also be considered for the non simple normal crossing (curves still should correspond to half-edges of the dual complex).
\end{definition}

\begin{definition} The duncehat complex is a triangle $(123)$ with edges $(12), (13)$ and $(23)$ all glued together.
\end{definition}

\textbf{Construct}. Let us consider a pair of nodal cubics on $\mathbb{P}^2$ in general position, and blow up $8$ points of their intersection. Also, blow up $1$ more point on the second cubic (it will force triple point formula). Let us set $X_1$ to be this $\mathbb{P}^2$ blown up in nine points, and normalizations of proper preimages of these cubics as $C_1$ and $C_2$. Let us denote by $p_1, p_2 \in C_1$ the preimages of the node of the first cubic, $q_1, q_2 \in C_2$ preimages of the node on the second cubic, $p_3 \in C_1, q_3 \in C_2$ the preimages of their only remaining intersection point. Let us set the gluing map $\phi: C_1 \rightarrow C_2$ by the images of $3$ points: $\phi(p_1) = q_2, \phi(p_2) = q_3, \phi(p_3) = q_1$.

\begin{center}
\includegraphics[trim={1cm 0cm 1cm 0cm}, clip, width=6cm]{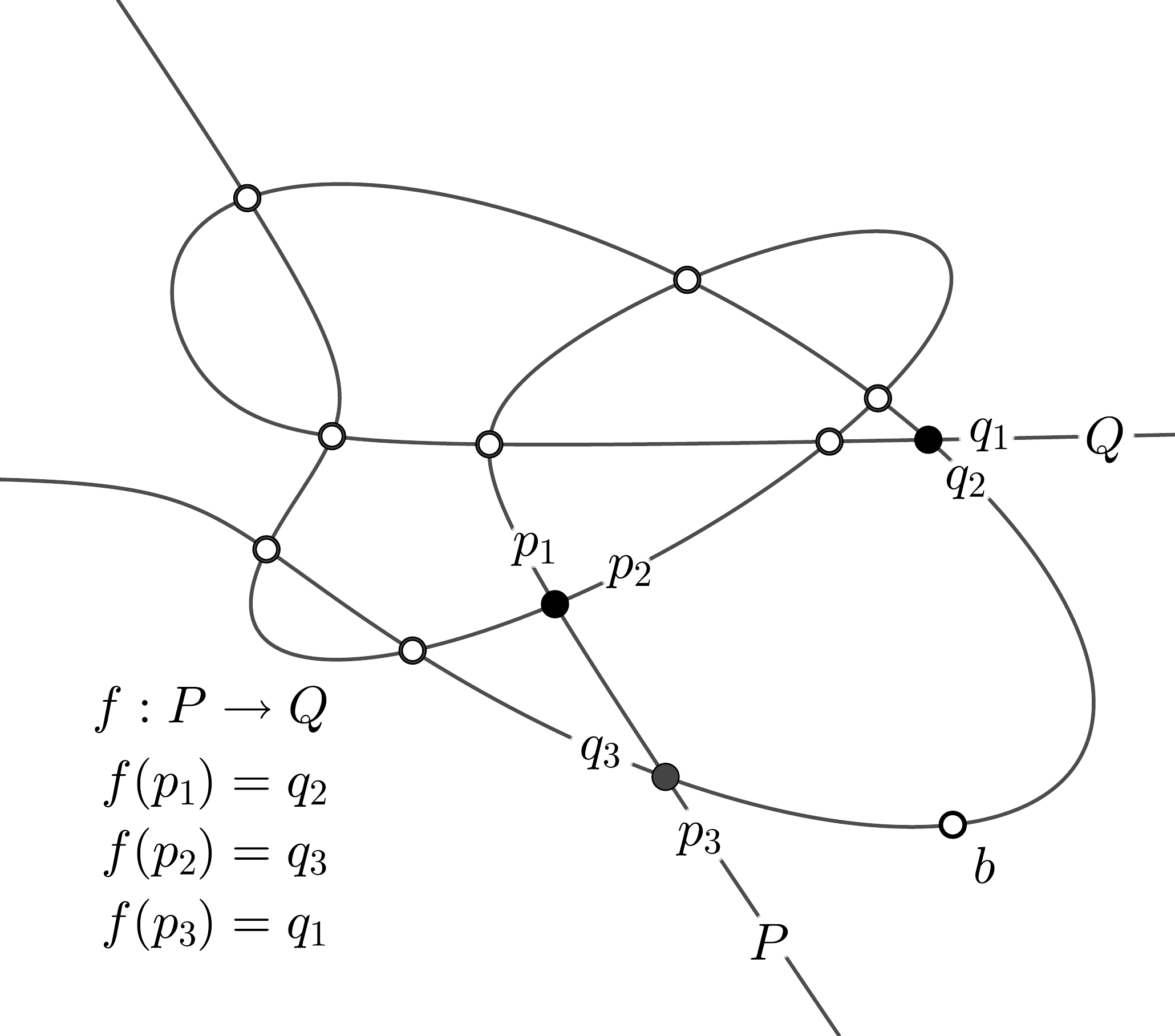}
\end{center}

\begin{proposition} Dual complex of this construct is a duncehat complex.
\end{proposition}

\begin{proof} It is more or less direct inspection to see that the resulting variety is normal crossing. It is also easy to see that triple point formula is satisfied. The only question remaining is the orientation of edges, which is defined from the cyclic orientation on the edges of the triple point. The following picture illustrates the situation around the triple point: \end{proof}
\begin{center}
\includegraphics[trim={1cm 0cm 1cm 0cm}, clip, width=6cm]{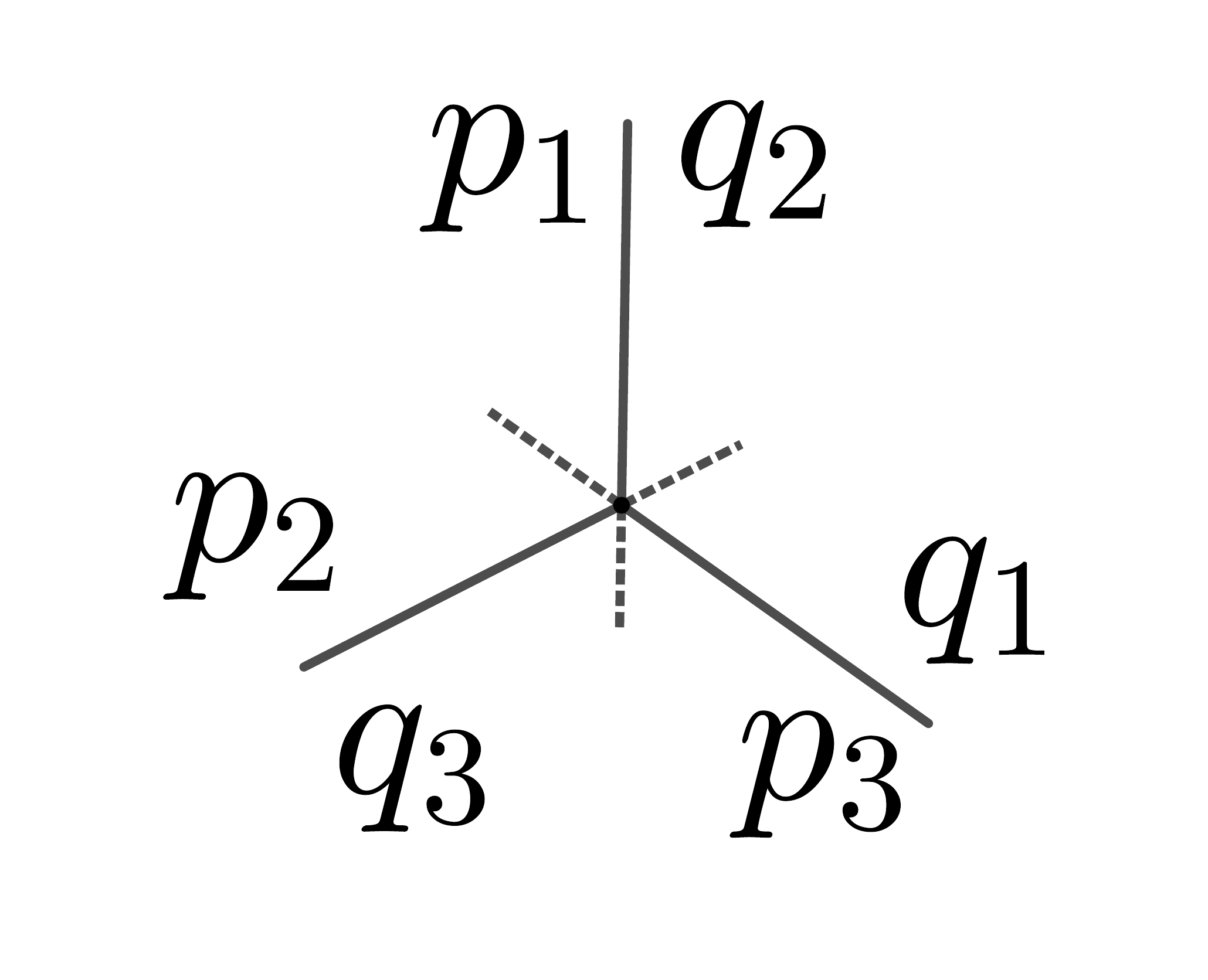}
\end{center}

The construct has $9$ parameters ($8$ coming from the relative positions of nodal cubics and $1$ coming from the position of additional blow up point). The $T_X^1$ bundle is a topologically trivial line bundle on a $\Sing(X)$, which is a rational curve with one triple self-intersection point. Moduli space of such bundles is $(\mathbb{C}^*)^2$, so naively one would expect to be able to force $d$-semistability by deforming the construct.

Calculating the mapping from the pairs of nodal cubics to the moduli of obstruction bundles directly is quite a tedious problem, so we instead describe its behaviour while approaching some codimension $1$ degenerations, and then use topological reasoning to prove the map is surjective. This approach should probably be restated (and greatly generalized) in terms of some compactifications of the relevant moduli spaces, but we did not manage to find out how to do it.

\section{Smoothing}
\subsection{Obstruction mapping}
We now make some preparations. Let us, as before, denote the construct as $X$. We assume that all gluing curves of $X$ have not more than three triple points (and if a gluing curve has less than three triple points we fix an additional marking points which are to be glued to the marking points). Let us denote the moduli space of configurations of $X$ as $M_X$. Under the assumption it automatically splits as the cartesian product: \[M_X = \prod M_{X_i}\]
where $M_{X_i}$ are moduli of configurations on each component.

We denote the singular locus $\Sing(X)$ as $S$ to emphasise that it actually doesn't depend on $X$.

Let us denote as $M_O$ the moduli space of obstruction bundles - i.e. the moduli space of line bundles on $S$ having the correct topological type. There is an obstruction mapping \[O: M_X \rightarrow M_O\] which is defined as follows: \[O(X) = (T_X^1)^{*}\]

In order to guarantee d-semistability it is enough to check that this mapping is surjective (or at least the preimage of the trivial bundle is nonzero).

We will use the following concrete description of $(T_X^1)^{*}$: its restriction on each component $C \in \Sing(X)$ is equivalent to $N^*_{C, X_i} \otimes N^*_{C, X_j} (- \sum_k p_k)$ where $X_i$, $X_j$ are analytic germs of components $C$ resides in (it could be the same component but different branches), and $\sum_k p_k$ is taken over triple points on $C$. The geometric fiber over the triple point $p$ lying on the intersection of curves $C_1$, $C_2$ and $C_3$ is canonically identified with $T^*_p(C_1) \otimes T^*_p(C_2) \otimes T^*_p(C_3)$. The bundles over different components are glued according to this identification.

For an element \[\omega = s \gamma_i \otimes \gamma_j \in N^*_{C, X_i} \otimes N^*_{C, X_j} (- \sum_k p_k)\] such that \[s \in \mathcal{O}(- \sum_k p_k) \]
\[\gamma_i \in N^*_{C, X_i}\]
\[\gamma_j \in N^*_{C, X_j}\]

let us denote by \[v_p (\omega) = ds \otimes \gamma_i \otimes \gamma_j\] the value of $\omega$ in the space $T^*_p(C) \otimes T^*_p(C_i) \otimes T^*_p(C_j)$, where $p \in C$ - triple point, $C_i$, $C_j$ - germs of other curves going through $p$ along $X_i$ and $X_j$ respectively. We will also sometimes denote by $v_p$ a number in a case the space $T^*_p(C) \otimes T^*_p(C_i) \otimes T^*_p(C_j)$ is trivialized.

Suppose now we are given (for all pairs (curve $C$, germ $X_i$ of the surface going through $C$)) a collection of meromorphic conormal forms \[\gamma_{C, X_i} \in N^*_{C, X_i}\] with zeros and poles away from the set of triple points. Let us also fix on each curve $C$ the function $s_C$ with zeros in the triple points and poles away from the set of triple points (say, for three triple points one can put \[s = \frac{x(x-1)}{(x-r)^3}\] where $(0, 1, \infty)$ are identified with triple points on a curve $C$, and $r$ is an arbitrary point).

For any point $p$ with germs of curves $C_1, C_2, C_3$ passing through it we choose the following trivialization:
\[ds_{C_1}|_p \otimes ds_{C_2}|_p \otimes ds_{C_3}|_p \in T^*_p(C_1) \otimes T^*_p(C_2) \otimes T^*_p(C_3)\]

Now, let us define the collection of meromorphic sections of $(T_X^1)^*$ on components.

\[\omega_C = s_c \gamma_{C,X_i} \otimes \gamma_{C, X_j}\]

it is easy to see that

\[v_p (\omega_C) = \frac{\gamma_{C, X_i}}{ds_{C_i}|_p} \frac{\gamma_{C, X_j}}{ds_{C_j}|_p}\]

Now, the line bundle $(T_X^1)^*$ admits the following concrete description in terms of this data: start off with the collection of trivial bundles on the components. Use the collection of identifications $v_p(\omega_C)$ to glue them together (the image of the section $1 \in \Gamma(\mathcal{O}_C)$ in the point $p$ should be $v_p(\omega_C)$). Denote the resulting bundle as $L$. Then,

\[(T_X^1)^* \simeq L([\omega])\]

This will be used in the next part to describe the asymptotic behavior of $(T_X^1)^*$ over various degenerations of the construct. Due to construct being normal crossing yet very not simple, there is quite a tedious case-by-case examination, so we just end the section with one particular case which is quite general and will be used later.

Suppose there is a family of constructs $X^z$ over the base with parameter $z$, such that in the central fiber one of the curves $C$ on one of the components $X_i$ degenerates into nodal rational curve $Q \cup Q'$. Suppose that the limits of the triple points $p_1, p_2$ land on $Q$, and limit of the triple point $p_3$ lands on $Q'$. We will now describe the asymptotic behavior of $(T_X^1)^*$ in this situation. \footnote{in what follows we will denote by the expression $q \sim z^k$ the fact that some quantity $q$ divided by $z^k$ has a finite nonzero limit for $z=0$}.

We denote by $\tilde{C}$ the family of curves $C$ over the disk $D$ with coordinate $z$. $\tilde{C}$ can be identified with $\mathbb{P}^1 \times D$, blown up in the point $p_3$ in the central fiber.

\begin{remark} It is possible to choose the conormal forms $\gamma$ in such a way that they have a well-defined limit for $z=0$, and their divisors are still away from the triple points. The conormal form $\gamma_{C, X_i}$ will split into conormal forms $\gamma_Q$ and $\gamma_{Q'}$, which will also both have a zero in the point of intersection $Q \cap Q'$.
\end{remark}

\begin{proof} It follows from the fact that the pullback of $N^*(C)$ is a line bundle over $\tilde{C}$, and the fact that $N^*(C_0) \rightarrow N^*{Q} \oplus N^*{Q'}$ has a zero of order $1$ in the node.
\end{proof}

Let us fix some (smooth riemannian) metric $g$ on $X^z$. We will use it to discuss asymptotic behavior of divisor $[\omega]$. Note that the fibers of $\tilde{C}$ are immersed in $X^z$ near the triple points, so the any smooth riemannian metric on $\tilde{C}$ is fiberwise continuous in $g$ and bilipschitz near the triple points, so asymptotic behavior can also be discussed in terms of internal geometry of $\tilde{C}$.

We will call the ''internal'' parametrisation of $C$ the parametrisation given by a triple $(p_1, p_2, p_3)$, and the corresponding Fubini-Study metric ''internal'', and the metric $g$ restricted on $C$ will be called ''external''. Recall that the function $s$ is defined internally.

\begin{lemma} The behavior of divisor $[\gamma_{C, X}]$ on $\tilde{C}$ under the $z \rightarrow 0$ is the following: it splits into two divisors $q$ and $q'$, landing on the components $Q$ and $Q'$ respectively. It means that in the internal parametrisation $q$ has a finite limit, and the distance between points of $q'$ and $p_3$ has asymptotic $\sim z$.
\end{lemma}

\begin{proof}
It is readily seen from the blow-up model for $\tilde{C}$ - before the blow-up the fibers are parametrized internally, and $q$ intersects the fiber transversely, while $q'$ intersects the exceptional divisor transversely away from $Q \cap Q'$ hence, tends to $p_3$ with asymptotic $\sim z$.
\end{proof}

We now proceed to regularize $\gamma_{C, X}$: choose the function $f$ in such a way that it has zeroes in $q'$ and poles in some divisor having a limit away from triple points. Then, we define

\[\gamma_{C, X}^{\reg} = \frac{\gamma_{C,X}}{f}\]

\begin{lemma}
\[|\gamma_{C,X}^{\reg}|_{p_1} \sim |\gamma_{C,X}^{\reg}|_{p_2} \sim 1\]
\[|\gamma_{C,X}^{\reg}|_{p_3} \sim z^{-ord(q')} = z^{\langle Q', Q' \rangle + 1}\]
\end{lemma}

\begin{proof}
The conormal forms $\gamma$ have a finite nonzero limit norm in the triple points. The first statement follows from the fact that $f$ is $\sim 1$ in the points $p_1$ and $p_2$, and the second statement is $f \sim z^{ord(q')}$, which is also clear.
\end{proof}

\begin{lemma}
The norm of $ds$ undergoes the following changes: on every curve except $C$ it has a finite nonzero $g$-norm, and on $C$

\[|ds_C|_{p_1} \sim |ds_C|_{p_2} \sim 1\]
\[|ds_C|_{p_3} \sim z\]
\end{lemma}

\begin{proof}
Recall that internal metric on $C$ is bilipschitz to the external metric on $Q$. This proves the first statement. For the second one, consider the function $s(x, z) = \frac{x(x-1)}{(x-r)^3}$ and blow up a point $(0, 0)$ (here we put $p_3 = 0$). The resulting function in the coordinates $(\tilde{x}, z) = (\frac{x}{z}, z)$ looks like \[\frac{z \tilde{x} (z \tilde{x} - 1)}{(z \tilde{x} - r)^3}\]

Recall that $\tilde{x}$ is bilipschitz to the external parameter on $Q'$. It is readily seen that \[\frac{ds}{d\tilde{x}}|_{p_3} \sim z\]
\end{proof}

Denote by \[\omega^{\reg} = \frac{\omega}{f}\] the collection of sections formed from $\gamma$ by the same rule as $\omega$ but with $\gamma_{C, X}$ changed to $\gamma_{C, X}^{\reg}$. Combining the statements above  we get the following

\begin{proposition}\label{bubble}
\[v_{p_3}(\omega_C^{\reg}) \sim z^{\langle Q', Q' \rangle + 1} \]
\[v_{p_3}(\omega_{C_{X}}) \sim z^{-1}\]

and other trivializations are $\sim 1$.
\end{proposition}

where $C_{X}$ denotes other curve going through $p_3$ on $X$.

The similar calculations can be performed for other simple degenerations of the construct.

\subsection{Asymptotic of $M_O$}

The additional blow-up point in the construction gives a very controllable way to change the $T_X^1$, so it makes sense to do the blow-up last. So, let us consider the space $M$ of pairs of nodal cubics in $\mathbb{P}^2$, blown up in $8$ of $9$ of their points of intersection and chosen branches in the nodes, and the space $M_O$ of line bundles on $\mathbb{P}^1 / (0, 1, \infty)$ of Chern number $-1$.

The space $M_O$ can be identified with $(\mathbb{C}^*)^3 / \mathbb{C}^*$, where the factors of the product are the fibers over $0, 1, \infty$. We put $M_O$ in a standard $\mathbb{P}^2 = \bar{M}_O$ in an obvious way.

\begin{proposition} \label{blowup}
Then, there is a quadric $Z \in \bar{M}_O$ of such bundles $L$ that there exists a unique point $x$ such that $L(x) \simeq \mathcal{O}$, passing through the coordinate points.
\end{proposition}

\begin{proof}
Let us choose the trivialization over the triple point (it amounts to the choice of the preimage in $(\mathbb{C}^*)^3$). Interpret $\mathcal{O}(-1)$ as a tautological bundle, then $L$ with chosen trivialization corresponds to a triple of points on a plane on lines corresponding to $0$, $1$ and $\infty$. The condition of existence of $x$ is the condition that these $3$ points lie on the same affine line, and it is obviously a quadric passing through coordinate points.
\end{proof}

The plan now is to prove that the image of the map $O: M \rightarrow M_O$ intersects $Z$. It will mean that there exists such a point that its blow-up will trivialize $T_X^1$. To prove this, we will consider a generic one-parameter family of pairs of nodal cubics (say, the second cubic will not move at all and first will deform in a generic way). This family will exhibit some degenerations, over which we carefully calculate the asymptotic behavior of $O$.

There are few cases that are to be considered.

\begin{itemize}
  \item \textbf{Case 0.} Nodal cubics are tangent to each other in a smooth point (not a chosen point of intersection). This is not even a degeneration - two successive blow-ups remove this intersection, so the construct is well defined for this configuration.
  \item \textbf{Case 1.} First of the nodal cubics degenerates into a line and a quadric. This has $4$ subcases, corresponding to whether $p_2$ and $p_3$ end up on the same component or $p_1$ and $p_3$, (\textbf{cases 1.1 and 1.2}) and, also, whether this component is a line or a quadric (\textbf{cases 1L and 1Q}).
  \item \textbf{Case 2.} First of the nodal cubics exhibits a cusp.
  \item \textbf{Case 3.} One of the cubics passes through a node of the other. This case has $3$ subcases (technically $6$ but interchanging $P$ and $Q$ removes $3$ of them). Either a node of a cubic lands on other cubic away from the chosen intersection point $3$ (\textbf{case 3.0}) or $p_1$ tends to $p_3$, \textbf{case 3.1} or $p_2$ tends to $p_3$ \textbf{case 3.2}.
\end{itemize}

\textbf{Case 0.} No degeneration.

\textbf{Case 1.} Proposition \ref{bubble} is directly applicable. The subcases are: either $p_2$ and $p_3$ land on a same component or $p_1$ and $p_3$ and whether this component is a line or a quadric. Let us fix the following notation. The second cubic will be called $Q$, and the components of the first cubic will be called $P_1$ and $P_2$. The component containing $p_3$ and another point will be called $P_1$, and another one will be called $P_2$. Now, whether $P_1$ is a line or a quadric, $\langle P_1, P_1 \rangle = -1$: for a line, it will be blown up in $2$ points (intersections with $Q$ which are not $p_3$), and quadric will be blown up in $5$ points (for the same reason). $\langle P_2, P_2 \rangle = -2$. So, the second case distinction doesn't matter - the proposition \ref{bubble} will be applied the same. The first distinction also doesn't matter: for example, in case 1.1

\[
v_1(\omega^{\reg}) = \frac{\gamma_P^{\reg}(p_1)}{ds_P(p_2)}\frac{\gamma_Q^{\reg}(q_2)}{ds_Q(q_1)} \sim \frac{z^{-1}}{1}\frac{1}{1} = z^{-1}
\]

\[
v_2(\omega^{\reg}) = \frac{\gamma_P^{\reg}(p_2)}{ds_P(p_1)}\frac{\gamma_Q^{\reg}(q_3)}{ds_P(p_3)} \sim \frac{1}{z} \frac{1}{1} = z^{-1}
\]

\[
v_3(\omega^{\reg}) = \frac{\gamma_P^{\reg}(p_3)}{ds_Q(q_3)}\frac{\gamma_Q^{\reg}(q_1)}{ds_Q(q_2)} \sim \frac{1}{1}\frac{1}{1} = 1
\]

and in the case 1.2

\[
v_1(\omega^{\reg}) = \frac{\gamma_P^{\reg}(p_1)}{ds_P(p_2)}\frac{\gamma_Q^{\reg}(q_2)}{ds_Q(q_1)} \sim \frac{1}{z}\frac{1}{1} = z^{-1}
\]

\[
v_2(\omega^{\reg}) = \frac{\gamma_P^{\reg}(p_2)}{ds_P(p_1)}\frac{\gamma_Q^{\reg}(q_3)}{ds_P(p_3)} \sim \frac{z^{-1}}{1} \frac{1}{1} = z^{-1}
\]

\[
v_3(\omega^{\reg}) = \frac{\gamma_P^{\reg}(p_3)}{ds_Q(q_3)}\frac{\gamma_Q^{\reg}(q_1)}{ds_Q(q_2)} \sim \frac{1}{1}\frac{1}{1} = 1
\]

\includegraphics[trim={1cm 0cm 1cm 0cm}, clip, width=6cm]{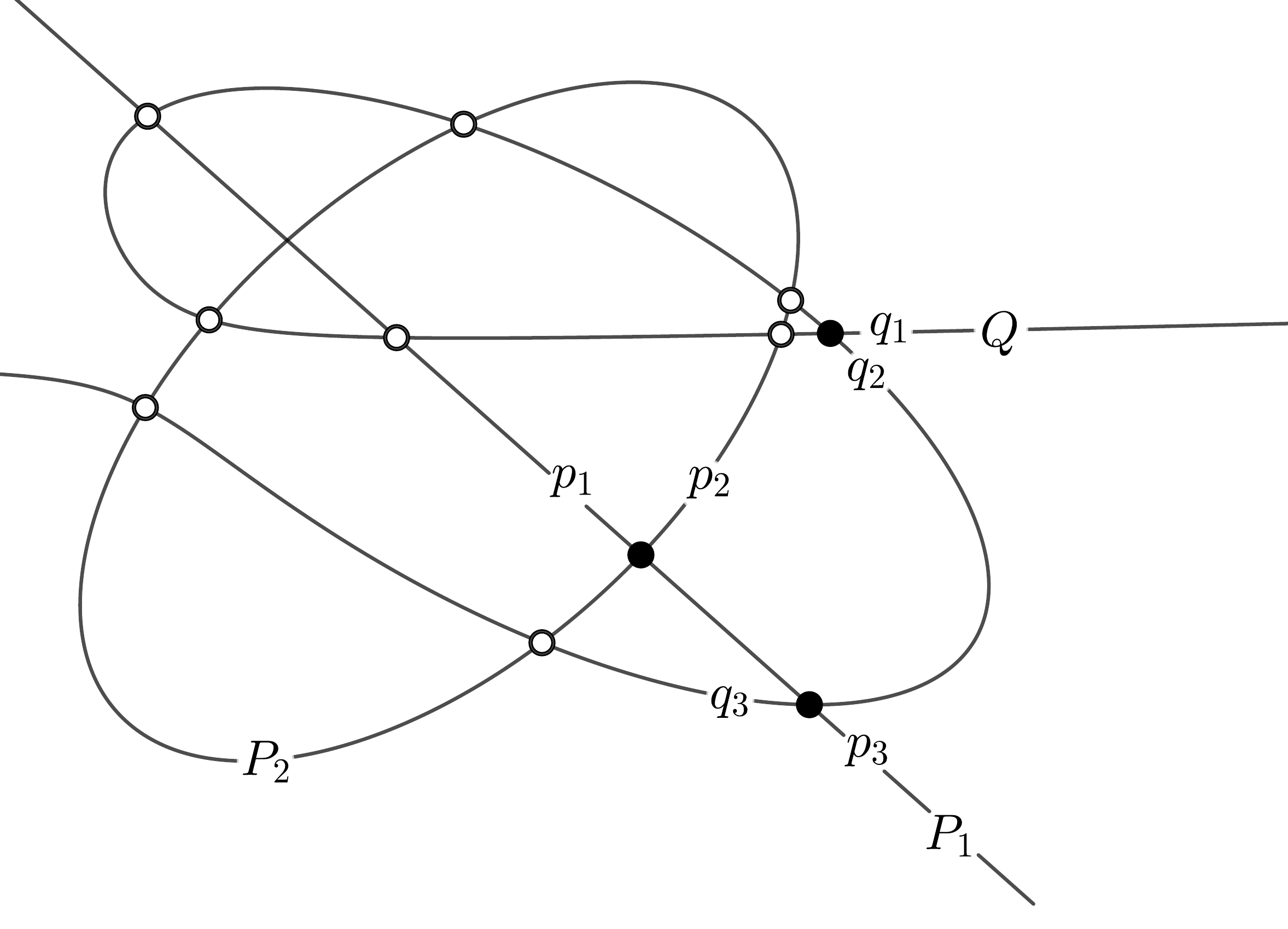}
\includegraphics[trim={1cm 0cm 1cm 0cm}, clip, width=6cm]{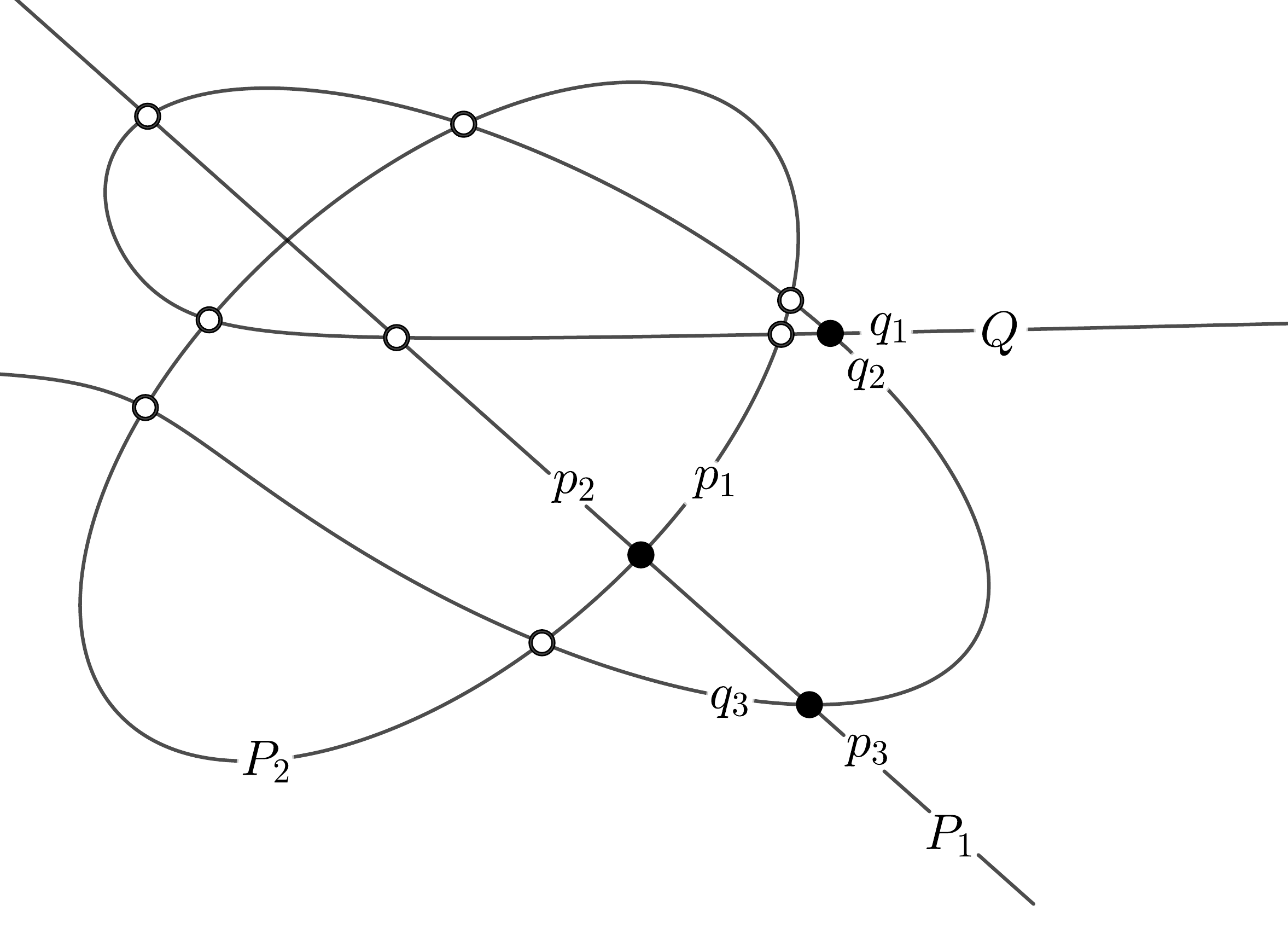}

\includegraphics[trim={1cm 0cm 1cm 0cm}, clip, width=6cm]{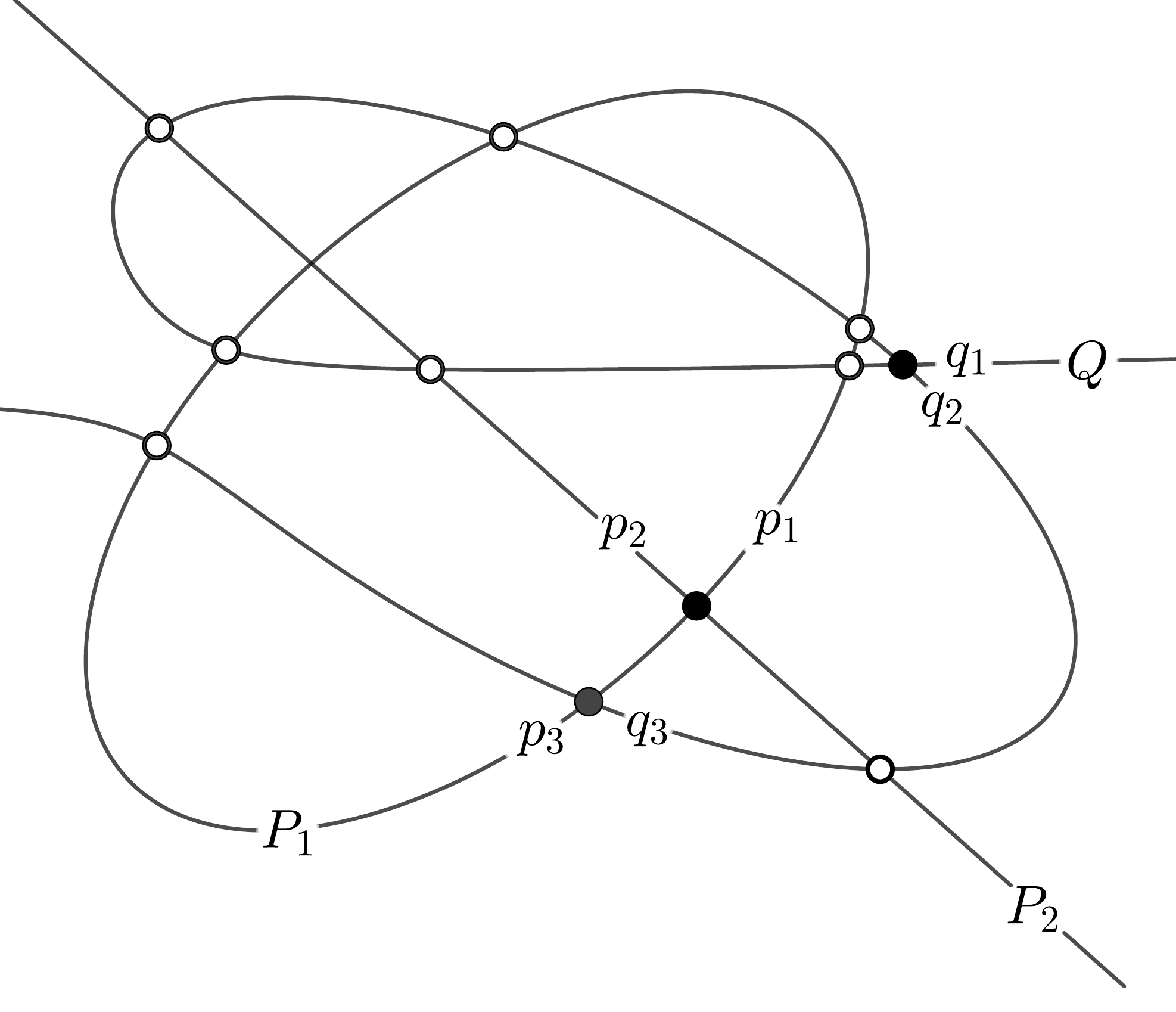}
\includegraphics[trim={1cm 0cm 1cm 0cm}, clip, width=6cm]{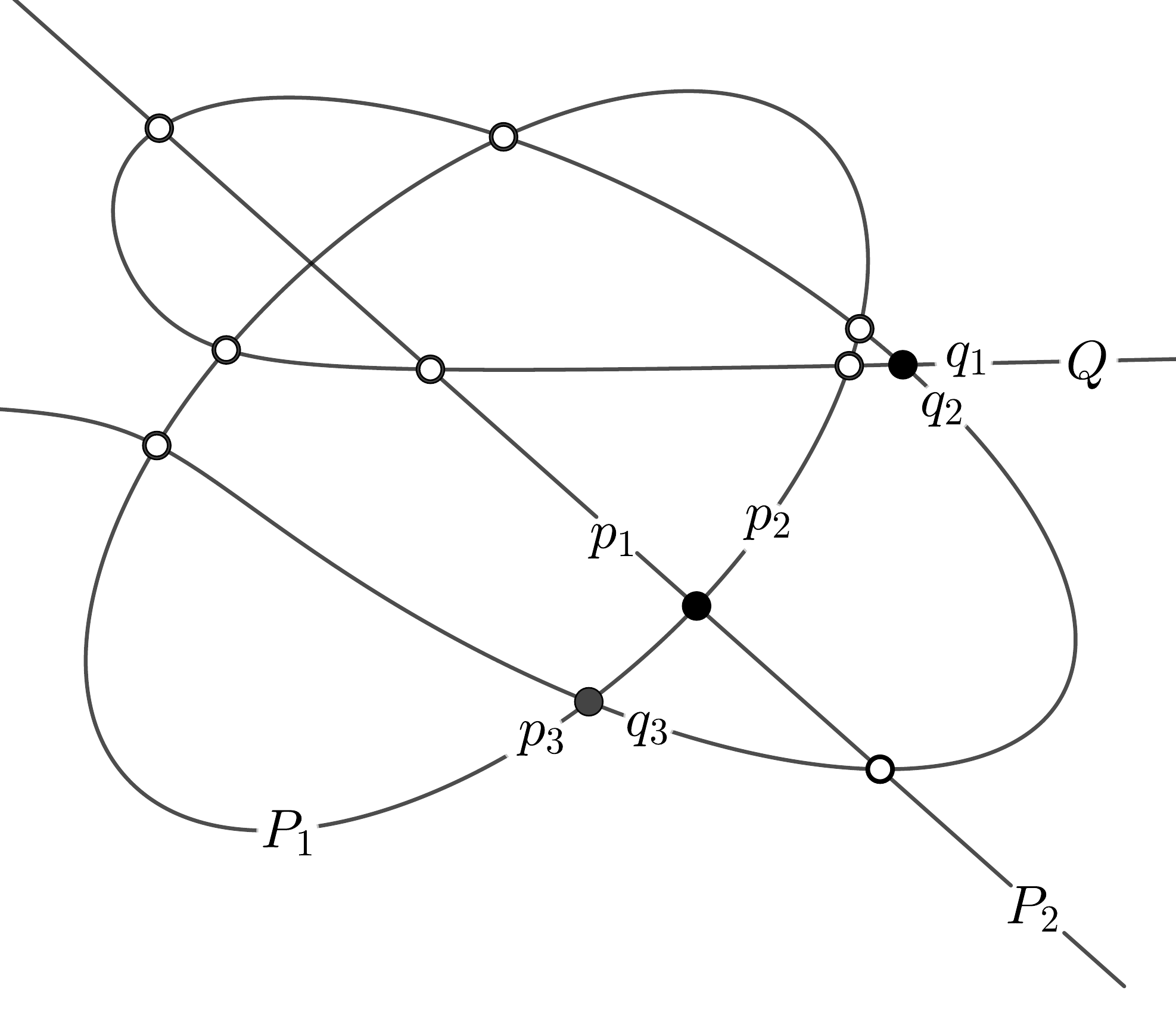}

\textbf{Case 2.} In a case when a cubic $P$ exhibits a cusp we need to make a natural choice of the conormal form $\gamma$, such that its divisor will have a limit for the cuspidal cubic. For the equation $f_P = y^2 - x^3 - z^2 x^2$ the obvious choice is the form \[\gamma_P = \frac{df}{y} = 2dy - \frac{(3x + 2z^2)x}{y} dx = 2dy - \frac{3x+2z^2}{\sqrt{x+z^2}} dx \]

It has a limit for a cuspidal cubic and is also well-defined externally (because Weierstrass normal form is well-defined externally). Similar to blow-up case, a family of nodal cubics degenerating into cusp should be though of as a map from the family given by the equation $f_p$ in $\mathbb{P}^2 \times D$. Then,

Consider also the parametrization $(x = l^2 - z^2, y = l(l^2 - z^2))$. The points $p_1$ and $p_2$ have parameter $l = \pm z$. The tangent vectors in this parametrization are $(z, \pm z^2)$. $\gamma_P$ evaluated on such a vector gives $\pm 4z^2$, $ds$ on such a vector is $\sim z^{-1}$. Also, $\gamma_P$ is to be regularized - it has a pole in the point $(-z^2, 0)$, and the rest of the divisor (which is of total degree $2$) tends to $p_3$.

\[|\gamma_P^{\reg}|_{p_1} \sim |\gamma_P^{\reg}|_{p_2} \sim z^2\]
\[|\gamma_P^{\reg}|_{p_3} \sim z^{-2}\]

\[|ds_P|_{p_1} \sim |ds_P|_{p_2} \sim z^{-1}\]
\[|ds_P|_{p_3} \sim z\]

Now, it is clear that
\[v_1 (\omega^{\reg}) = \frac{\gamma_P^{\reg}(p_1)}{ds_P(p_2)}\frac{\gamma_Q^{\reg}(q_2)}{ds_Q(q_1)} \sim \frac{z^2}{z^{-1}}\frac{1}{1} = z^{3}\]
\[v_2 (\omega^{\reg}) = \frac{\gamma_P^{\reg}(p_2)}{ds_P(p_1)}\frac{\gamma_Q^{\reg}(q_3)}{ds_P(p_3)} \sim \frac{z^2}{z^{-1}}\frac{1}{z} = z^{2}\]
\[v_3 (\omega^{\reg}) = \frac{\gamma_P^{\reg}(p_3)}{ds_Q(q_3)}\frac{\gamma_Q^{\reg}(q_1)}{ds_Q(q_2)} \sim \frac{z^{-2}}{1}\frac{1}{1} = z^{-2}\]

\begin{center}
\includegraphics[trim={1cm 0cm 1cm 0cm}, clip, width=6cm]{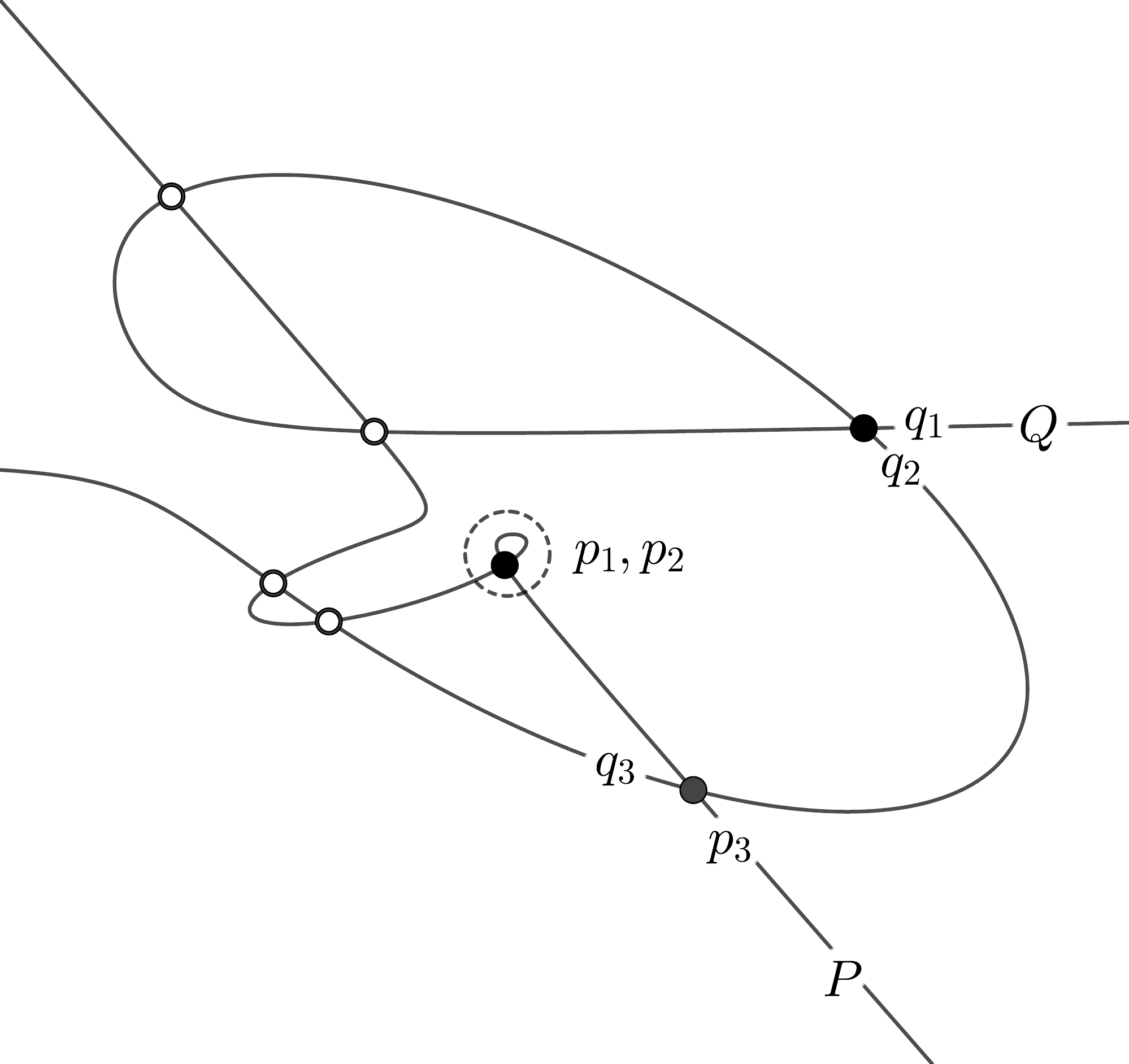}
\end{center}

\textbf{Case 3} In the third case, the conormal forms $\gamma$ will be taken as any generic external meromorphic forms, however, some of their zeros will tend to the triple points and will need to be regularized, too. Consider at first the case 3.0, where the nodal point of $P$ prepares to pass through $Q$ away from the point $3$. Denote the intersections of the branches of $P$ near the node with $Q$ as $b_1$ and $b_2$ respectively. These points are subject to blow-up, hence form a part of divisor $[\gamma_P]$ which tends to triple points. The norms of $ds$ all have the finite limit (neither of cubics with the collection of three points exhibits a degeneration). The only ingredient of $\omega$ which has some asymptotic behavior is $\gamma_P$, who has two of its roots tending to $p_1$ and $p_2$ respectively. After regularization, we get

\[|\gamma_P^{\reg}|_{p_1} \sim z^{-1}\]
\[|\gamma_P^{\reg}|_{p_2} \sim z^{-1}\]
\[|\gamma_P^{\reg}|_{p_3} \sim 1 \]

which leads to asymptotic

\[v_1 (\omega^{\reg}) = \frac{\gamma_P^{\reg}(p_1)}{ds_P(p_2)}\frac{\gamma_Q^{\reg}(q_2)}{ds_Q(q_1)} \sim z^{-1}\]
\[v_2 (\omega^{\reg}) = \frac{\gamma_P^{\reg}(p_2)}{ds_P(p_1)}\frac{\gamma_Q^{\reg}(q_3)}{ds_P(p_3)} \sim z^{-1}\]
\[v_3 (\omega^{\reg}) = \frac{\gamma_P^{\reg}(p_3)}{ds_Q(q_3)}\frac{\gamma_Q^{\reg}(q_1)}{ds_Q(q_2)} \sim = 1\]

by the same argument as before (point $1$ neighboring $p_2$, and point $2$ neighboring $p_1$).

\begin{center}
\includegraphics[trim={1cm 0cm 1cm 0cm}, clip, width=6cm]{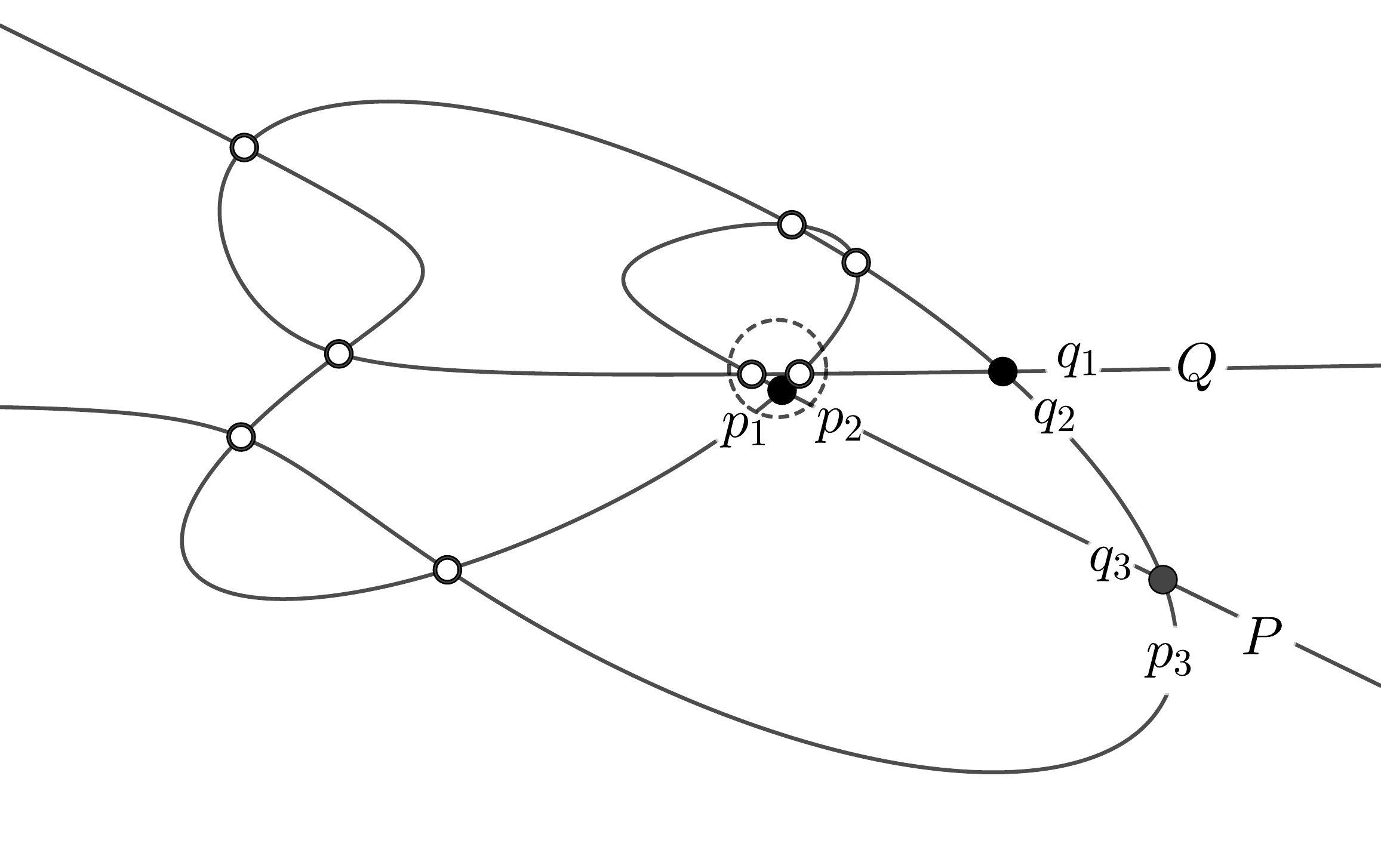}
\end{center}

Cases 3.1 and 3.2 are a bit different. Take as the parameter $z$ the difference between ($p_1$ and $p_3$ or $p_2$ and $p_3$, respectively)
In the case 3.1

\[|ds_P|_{p_1} \sim z^{-1}\]
\[|ds_P|_{p_2} \sim z\]
\[|ds_P|_{p_3} \sim z^{-1}\]

\begin{proof} Parametrize the cubic $P$ externally (say, standard rational parametrization by projection on a fixed line). Then, external parametrization of $P$ exhibits a finite limit, and in it $\dist(p1, p3) \sim z$. Then, $|ds|$ in a triple point is proportional to the Jacobian of the rational map from external parametrization to internal.
\end{proof}

Denote the nearby point of intersection of the branch of $p_2$ with $Q$ as $m$. The divisor $[\gamma_{P}] - m$ tends to $p_2$ with $\sim z$, and $m$ tends with $\sim z^2$ (due to it actually moving with $\sim z$ towards $p_2$ in external parametrization). The degree of $[\gamma_P]$ is $1$, so
\[\gamma_P^{\reg}|_{p_2} \sim z^{-2}\]

also, $\gamma_Q$ now also requires regularization - the divisor $[\gamma_Q] - m$ has a finite limit, and $m$ tends to $q_3$ with $\sim z$. Hence,

\[\gamma_Q^{\reg}|_{q_3} \sim z^{-1}\]

we are now in position to calculate the asymptotic.

\[v_1 (\omega^{\reg}) = \frac{\gamma_P^{\reg}(p_1)}{ds_P(p_2)}\frac{\gamma_Q^{\reg}(q_2)}{ds_Q(q_1)} \sim \frac{1}{z}\frac{1}{1} = z^{-1}\]
\[v_2 (\omega^{\reg}) = \frac{\gamma_P^{\reg}(p_2)}{ds_P(p_1)}\frac{\gamma_Q^{\reg}(q_3)}{ds_P(p_3)} \sim \frac{z^{-2}}{z^{-1}}\frac{z^{-1}}{z^{-1}} = z^{-1}\]
\[v_3 (\omega^{\reg}) = \frac{\gamma_P^{\reg}(p_3)}{ds_Q(q_3)}\frac{\gamma_Q^{\reg}(q_1)}{ds_Q(q_2)} \sim \frac{1}{1}\frac{1}{1} = 1\]

Now to the case 3.2.

In the case 3.2
\[|ds_P|_{p_1} \sim z\]
\[|ds_P|_{p_2} \sim z^{-1}\]
\[|ds_P|_{p_3} \sim z^{-1}\]

(same argument, but now $p_2$ tends to $p_3$)

Again, denote the nearby point of blow-up as $m$. Now, divisor $[\gamma_P]-m$ tends to $p_1$ with $\sim z$, yet still has degree $0$, and $m$ tends to $p_1$: points $p_2$ and $p_1$ switch roles on the component $P$, but on the component $Q$ situation is the same - $m$ still tends to $q_3$ with $\sim z$. So,

\[\gamma_P^{\reg}|_{p_1} \sim z^{-2}\]

\[\gamma_Q^{\reg}|_{q_3} \sim z^{-1}\]

which leads to

\[v_1 (\omega^{\reg}) = \frac{\gamma_P^{\reg}(p_1)}{ds_P(p_2)}\frac{\gamma_Q^{\reg}(q_2)}{ds_Q(q_1)} \sim \frac{z^{-2}}{z^{-1}}\frac{1}{1} = z^{-1}\]
\[v_2 (\omega^{\reg}) = \frac{\gamma_P^{\reg}(p_2)}{ds_P(p_1)}\frac{\gamma_Q^{\reg}(q_3)}{ds_P(p_3)} \sim \frac{1}{z}\frac{z^{-1}}{z^{-1}} = z^{-1}\]
\[v_3 (\omega^{\reg}) = \frac{\gamma_P^{\reg}(p_3)}{ds_Q(q_3)}\frac{\gamma_Q^{\reg}(q_1)}{ds_Q(q_2)} \sim \frac{1}{1}\frac{1}{1} = 1\]

\includegraphics[trim={1cm 0cm 1cm 0cm}, clip, width=6cm]{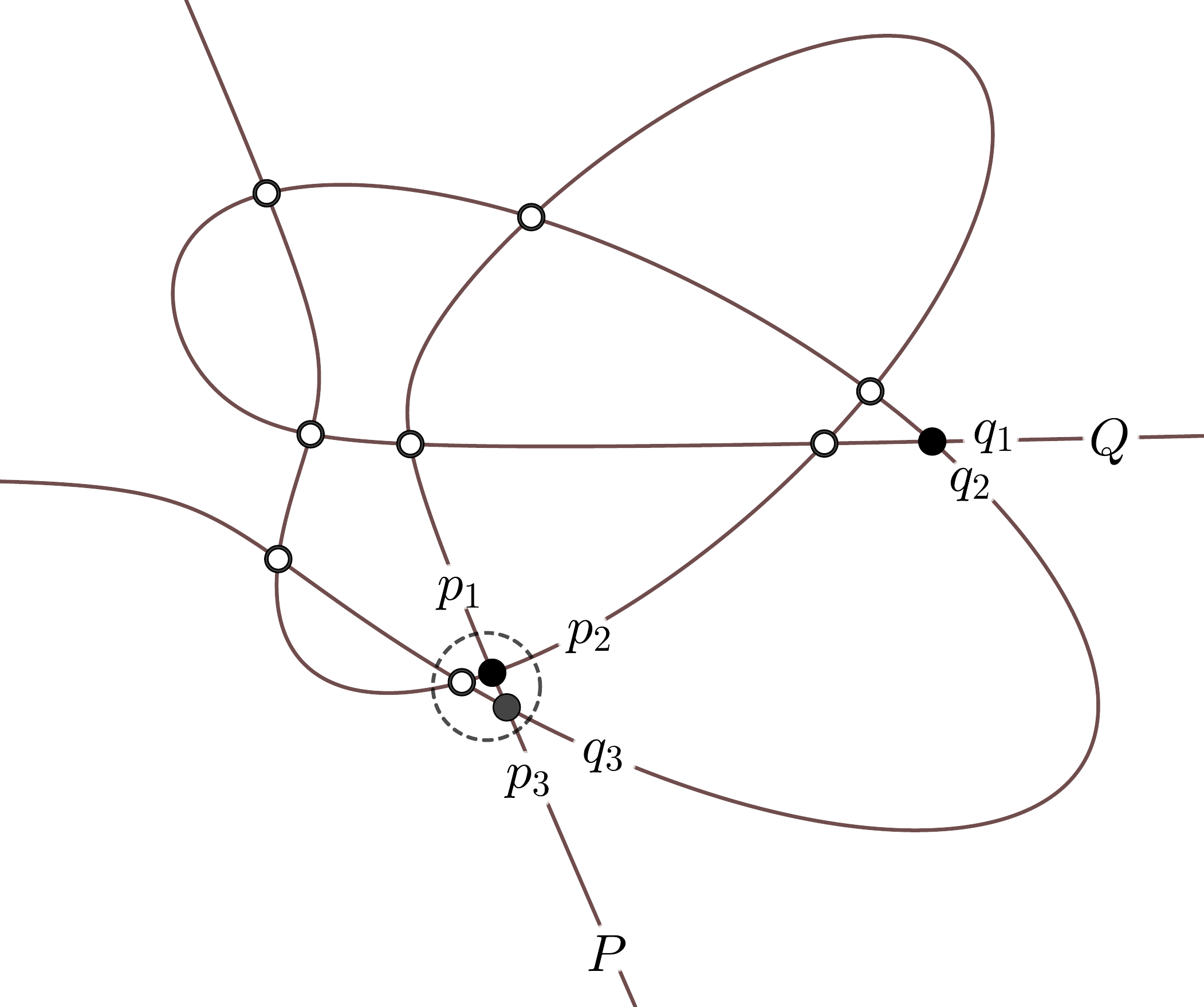}
\includegraphics[trim={1cm 0cm 1cm 0cm}, clip, width=6cm]{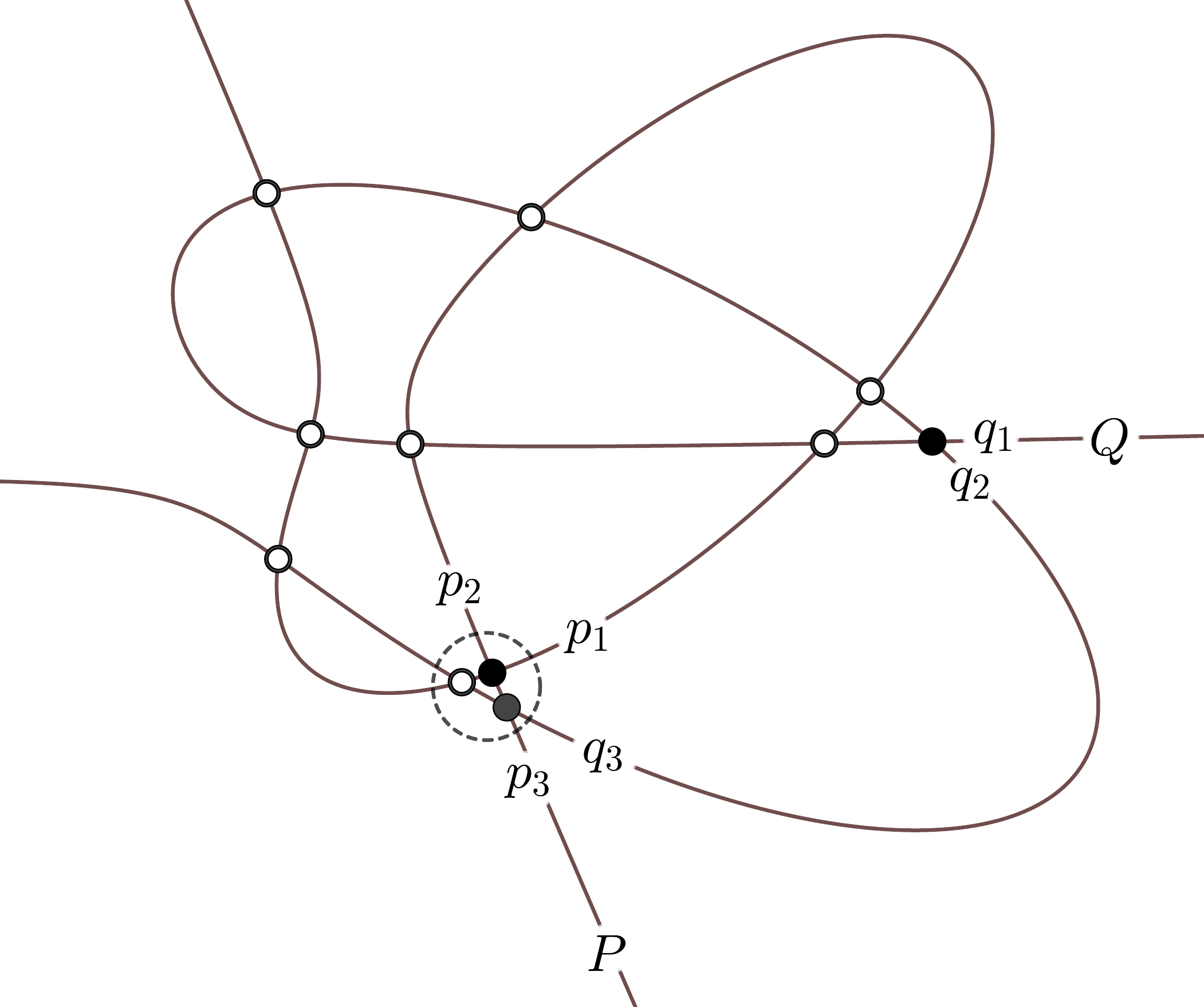}

\begin{theorem} Consider the $1$-parametric family in $M$ which only exhibits degenerations above. Then, its image (which we will denote $C$) intersects the quadric $Z$ away from the degeneration.
\end{theorem}

\begin{proof}
All degenerations except the cuspidal are automatically away from $Z$, because their image in $\bar{M_O}$ is on coordinate lines, but not in the coordinate points. Denote by $T$ the union of three lines. Every cuspidal degeneration corresponds to the intersection of the family with $T$ of order $9$ (because it tends along the toric orbit $(z^3, z^2, z^{-2}) = (z^5, z^4, 1)$), and it also corresponds to the intersection of order $4$ with $Z$. But intersection of $Z$ with $C$ is $2 \deg(C)$, and with $T$ it is $3 \deg(C)$. Suppose all intersections of $Z$ with $C$ would be in the cuspidal degeneration. It would lead to at least $\frac{9}{4} 2 \deg(C) = \frac{9}{2} \deg(C) > 3 \deg(C)$ intersections with $T$ (taken with multiplicity), which is absurd.
\end{proof}

\begin{corollary}
There exists a d-semistable construct. Moreover, the fiber of $O$ over the trivial $T_X^1$ (space of d-semistable constructs) has the expected dimension $7$
\end{corollary}

\begin{proof}
Consider generic $1$-parametric family in $M$, and choose intersection of its image with $Z$ (for example, pick a generic plane in the projective space of all cubic curves and choose all degenerate ones, then consider cover corresponding to branch choices). Moreover, because it intersects $Z$ transversely, the preimage of $Z$ in $M$ has the dimension $7$ (if it would be $8$ image of every family would lie in $Z$ due to $M$ being irreducible). But locus $O^{-1}(Z) \in M$ is identified with the space of d-semistable constructs by the proposition $\ref{blowup}$.
\end{proof}

\begin{note} It is extremely important that the dimension of the locus of d-semistable constructs is expected, and the trick done here is very low-dimensional.
\end{note}

\subsection{Smoothing}

We are unable to use logarithmic deformation theory directly: while it is easy to calculate that $h^1(X, T_X(log)) > h^2(X, T_X(log))$, it doesn't seem to be a sufficient condition for the existence of the logarithmic deformations. We were able to show that $h^2(X, T_X(log)) = 0$ in case the d-semistable locus $O^{-1}(1) \in M_X$ is smooth (so, it is enough to prove that it is reduced then it will have a smooth point) but this statement in itself seems to be very tedious to prove.

Instead, we make use of the full deformations, similar to original Friedman's approach. From now on, we will denote $T_X$ as $T_X^0$ to emphasise that it is only a $0$th local cohomology of the tangent complex.

We start off by showing that $H^2(X, T_X^0) = 0$. Denote by $\bar{X}$ the normalization of the construct, by $D = \bigcup_i D_i = P \cup Q$ the gluing locus, by $\pi: \bar{X} \rightarrow X$ the normalization (gluing) map.

\begin{lemma}
$H^2(T_{\bar{X}}(D)) = 0$
\end{lemma}

\begin{proof}
Let us first calculate $\chi(T_{\bar{X}}(D))$. It admits the following resolution:

\[0 \rightarrow T_{\bar{X}}(D) \rightarrow T_{\bar{X}} \rightarrow \bigoplus_i N_{D_i} \rightarrow 0\]

The strategy of our proof is backwards - we just describe the moduli space of a construct directly in geometric terms and then see that it is reduced, smooth and has the expected dimension ($ - \chi(T_{X} + (P^2 + 1) + (Q^2 + 1))$).

$M_X$ is clearly smooth and reduced, and has the dimension $9$. $X$ is $\mathbb{P}^2$ blown up in 9 points, so $\chi(T_X) = 8-2*9 = -10$. $P^2 = -2$, $Q^2 = -1$, and we obtain the desired expected dimension.
\end{proof}

Now, consider the sheaf $T^0_X$ on $X$. Its cohomologies control locally trivial deformations of the construct. The sheaf $T^0_X$ admits the following resolution:

\[0 \rightarrow T^0_X \rightarrow T_{\bar{X}}(D) \rightarrow \bigoplus_{C \in E} T_C (-\sum_{p \in T, C} p)  \rightarrow 0\]

The group $H^1(X, \bigoplus_{C \in E} T_C (-\sum_{p \in T, C} p))$ in case every curve in the construct has not more than three triple points. Applying the long exact sequence, we obtain that $h^2(X, T^0_X)$ of this sheaf vanishes, too. Moreover, the description of deformations of the construct has evident geometric sense:

\[H^0(X, T_{\bar{X}}(D)) \rightarrow H^0(X, \bigoplus_{C \in E} T_C (-\sum_{p \in T, C} p)) \rightarrow H^1(X, T_X) \]\[\rightarrow H^1(X,  T_{\bar{X}}(D)) \rightarrow H^1(X, \bigoplus_{C \in E} T_C (-\sum_{p \in T, C} p)) \rightarrow H^2(X, T_X) \rightarrow 0\]

First of all, $H^1(X, \bigoplus_{C \in E} T_C (-\sum_{p \in T, C} p)) = 0$, because all our curves are rational and have no more than $3$ triple points, hence these line bundles are of degree at least $-1$. It also makes sense from the deformation point of view: if there would be more than three points triple points, the it wouldn't be always possible to glue a construct from its components, and $H^2$ would control this (infinitesimal) obstruction.

\[H^0(X, T_{\bar{X}}(D)) \rightarrow H^0(X, \bigoplus_{C \in E} T_C (-\sum_{p \in T, C} p)) \rightarrow \]\[ \rightarrow H^1(X, T_X) \rightarrow H^1(X, T_{\bar{X}} (D)) \rightarrow 0\]

The remaining arrows in this diagram have straightforward geometric interpretation, too. From right to left:

\begin{itemize}
\item deformations of the construct yield deformations of its components, this map is surjective
\item deformations which do not change the components are controlled by the changes in the gluing data
\item not all changes in the gluing data are non-trivial - some correspond to the automorphisms of the components
\end{itemize}

The reference for the rest of discussion is Friedman's proposition 4.3 and around.

Consider the following local to global exact sequence for the tangent cohomology

\[H^1(T_X^0) \rightarrow T^1(X) \rightarrow H^0 (T_X^1) \rightarrow\]

\[H^2(T_X^0) \rightarrow T^2(X) \rightarrow H^1 (T_X^1) \rightarrow 0\]

it shows that $T^2(X) \simeq H^1(T_X^1)$ and $T^1$ is an extension of $H^0 (T_X^1)$ by $H^1(T_X^0)$ (it splits, but we will care about the structure of Massey operations on it so it is worth remembering that it is only a filtration, not a direct sum). Let us denote the coordinates in $T^1(X)$ as $a_1, ..., a_n, t$, ($n = 9$) projection to the $1$-dimensional space $H^0 (T_X^1)$ being $t$. Recall that the Kuranishi moduli is a fiber of an analytic mapping from $T^1(X)$ to $T^2(X) = H^1 (T_X^1)$ (so, formal subscheme given by $\dim(H^1 (T_X^1)) = 2$ analytic equations, say, call them $f_1, ..., f_k$, ($k = 2$)). Each of these equations is ought to be divisible by $t$ (because locally-trivial deformations are unobstructed, or we could say that Massey operations (which are terms of these equations) preserve the filtration). Denote by $\tilde{f_i} = \frac{f_i}{t}$.

\begin{theorem} Provided that the locus of d-semistable constructs has expected dimension, the solutions to the system $\tilde{f_i}$ are not contained in the scheme $t \neq 0$.
\end{theorem}

\begin{proof} First, it is clear that the solutions of $(\tilde{f_i} = 0, t = 0)$ correspond to the locally trivial deformations in the d-semistable locus - it can be readily seen from the following characterisation of d-semistable locus: ''such constructs that admit first order smoothings'' (it is equivalent to d-semistability provided $H^2(X, T^0_X) = 0$). The expected dimension of d-semistable locus is $n-d$ (n is the amount of locally trivial deformations, d is the amount of deformations of the obstruction bundle). Subspace $(\tilde{f_i} = 0)$ has dimension at least $n+1-d$, so it is not contained in $(\tilde{f_i} = 0, t = 0)$.
\end{proof}

\begin{note} The situation which could possibly occur if the d-semistable locus would too big is that one of the equations could be $t^r = 0$. It would prevent any deformation from lifting to the $r+1$'st order.
\end{note}

\begin{corollary} $X$ is smoothable.
\end{corollary}

\begin{proof} Consider the total space of the deformation constructed in the theorem above. If the deformation was transverse to the subspace $t = 0$, the total space is smooth. If the deformation was not transverse to $t=0$ (say, with tangency $k$), it means that it is locally trivial up to $k-1$'st order, and the image of its $k$'th Kodaira-Spencer differential is a non-trivial element in $H^0(X, T_X^1)$, which means that this deformation admits a local model of the form $x_1 x_2 x_3 = t^k$ (in particular, central fiber is a toroidal embedding).
\end{proof}

\section{Prospects and some calculations}

After proving that this surface is indeed smoothable, we can ask a few questions on some numerical invariants of this (smoothed) surface - obviously, the author knew these before actually proving it is smoothable, but the exposition turned out to not need these calculations. In this section to lower the amount of indices a bit, let us denote the smoothed surface as $\tilde{X}$.

The euler characteristic of $\tilde{X}$ equals the euler characteristic with compact support of $X \setminus \Sing(X)$. (it is true for any toroidal degeneration).

\[ \chi(\tilde{X}) = \chi^c(X \setminus \Sing(X)) = 12 - 1 = 11 \]

It is also easy to see that $K_X$ is big, hence the same holds for $\tilde{X}$, so it is a surface of general type.

Also, from $h^{1,0} (\tilde{X}) = h^{2,0} (\tilde{X}) = 0$ we can calculate that $h^{1, 1} = 9$.

On the basis of these calculations the author suspects that $\tilde{X}$ lies in a deformation class of the Barlow surface. Probably the cheapest way of doing this would be just constructing this degeneration. Saying anything useful about the fundamental group of $\tilde{X}$ (say, proving that it is $0$) also sounds reasonable.

Now we would like to discuss the meaning of these results and possible ramifications. The methods discussed above work in a situation where $H^2(T_X^0) = 0$ (which is a frequent occasion) and the d-semistable locus has expected dimension. This condition, however, seems hard to check. Possibly, the situation could be improved with some compactification of the space of constructs instead of the explicit calculation of asymptotics.

The arguments about the deformations should probably be recast in terms of logarithmic deformation theory - yet author didn't manage to significantly improve the result (say, guarantee the smoothability directly from $h^1(T_X(log)) > h^2(T_X(log))$).

Also, there is an infinite amount of non-collapsible complexes (even up to non-increasing dimension extension-collapse equivalence) and even bigger amount of complexes with point-like rational cohomology. Yet, there is only finite amount of deformation classes of surfaces of general type with $h^{1,0} = h^{2,0} = 0$. The discrepancy might have two sources. The first one is plain and simple - some complexes might not model any construct.

\begin{definition} \textbf{Combinatorial construct} is a simplicial complex with numbers on $n_{ij}$ on half edges, such that for any vertex $x$ the matrix of incidence of the link of $x$ with diagonal entries $n_{xi}$ has positive index of inertia at most $1$.
\end{definition}

This is clearly a needed restriction due to Hodge's index theorem. However, we have no idea whether this data is enough to build an actual construct (and if there even exists a simple combinatorial criterion for the graph of intersections of curves to be representable on a rational surface).

However, even this naive combimatorial restriction might significantly lower the space of possibilities. For example, the only $2$-dimensional surfaces admitting the structure of combinatorial construct are $\mathbb{S}^2, \mathbb{RP}^2, \mathbb{T}^2$. The proof (we learned it from \cite{KS}) can be derived from the fact that the structure of combinatorial construct will actually determine the integral-affine structure on the surface with positively charged singularities, and then just application of integral version of Gauss-Bonnet theorem.

\textbf{Hope/Wish.} Possibly, there is some notion of ''integral'' positive curvature which makes sense for complexes (determined either in terms of combinatorial construct or maybe some kind of tropical refinement of this notion), which cuts only a finite amount of complexes.

Yet, there is another possible source of discrepancy between the amount of surfaces of general type and amount of complexes. Namely, two entirely different (not expansion-collapse equivalent) complexes could very well be the degenerations of the same surface. The hope of circumventing this problem by ''going around the boundary of the moduli space and making some elementary rearrangements'' is too naive to work - even for curves there is Ezra Getzler's relation which does not amount to the elementary homotopy equivalence on the dual complex. However, author still keeps maybe even more naive hope that it is possible to somehow capture \textit{more} information from this complex than just the holomorphic part of the algebra of cohomology. We would imagine some mix between the theory of elementary collapses and theory of rational homotopy type, say, one could come up with the notion of elementary collapse of dg-algebra, stricter than quasi-isomorphism.

The other directly approachable question in this regard would probably be understanding something about the fundamental group.

\pagebreak


\begin{thebibliography}{50}

\bibitem{B} Rebecca Barlow, \textsl{Some new surfaces with $P_g=0$}, Duke Math. J., Volume 51, Number 4 (1984), 889-904.
\bibitem{dFKX} Tommaso de Fernex, János Kollár, and Chenyang Xu, \textsl{The dual complex of singularities}, Adv. Stud. Pure Math. Higher Dimensional Algebraic Geometry: In honour of Professor Yujiro Kawamata's sixtieth birthday, K. Oguiso, C. Birkar, S. Ishii and S. Takayama, eds. (Tokyo: Mathematical Society of Japan, 2017), 103 - 129
\bibitem{Fr}  Robert Friedman, \textsl{Global Smoothings of Varieties with Normal Crossings},  Annals of Mathematics
Second Series, Vol. 118, No. 1 (Jul., 1983), pp. 75-114
\bibitem{Fu} Takao Fujita, \textsl{On Del Pezzo fibrations over curves}, Osaka Math. J. 27 (1990) 229–245
\bibitem{GS} Mark Gross and Bernd Siebert, \textsl{Mirror Symmetry via Logarithmic Degeneration Data I}, J. Differential Geom. Volume 72, Number 2 (2006), 169-338.
\bibitem{Ka} Yasuyuki Kachi, \textsl{Global smoothings of degenerate Del Pezzo surfaces with normal crossings}, Journal of Algebra 307 (2007) 249–253
\bibitem{KN} Yujiro Kawamata, Yoshinori Namikawa, \textsl{Logarithmic deformations of normal crossing varieties and smoothing of degenerate Calabi-Yau varieties}, Inventiones mathematicae, December 1994, Volume 118, Issue 1, pp 395–409
\bibitem{KS} Maxim Kontsevich, Yan Soibelman, \textsl{Affine structures and non-archimedean analytic spaces}, arXiv:math/0406564v1  [math.AG]  28 Jun 2004
\bibitem{Ku} Viktor S. Kulikov, \textsl{Degenerations of K3 surfaces and Enriques surfaces}, Izv. Akad. Nauk SSSR Ser. Mat., 41:5 (1977), 1008–1042; Math. USSR-Izv., 11:5 (1977), 957–989
\bibitem{PP} Ulf Persson and Henry Pinkham, \textsl{Degeneration of Surfaces with Trivial Canonical Bundle}, Annals of Mathematics Second Series, Vol. 113, No. 1 (Jan., 1981), pp. 45-66
\bibitem{St} Joseph Steenbrink, \textsl{Limits of Hodge structures} Invent Math (1976) 31: 229. https://doi.org/10.1007/BF01403146
\bibitem{Tz} Nikolaos Tziolas, \textsl{Smoothings of Fano varieties with normal crossing singularities} , arXiv:1005.0531 [math.AG]
\end{thebibliography}
\end{document}